\documentclass[12pt]{amsart}
\usepackage{amssymb}
\usepackage[latin1]{inputenc}
\usepackage{amsthm}
\usepackage{amsfonts}
\usepackage{mathrsfs}
\usepackage{graphicx}
\usepackage{amsmath}
\usepackage[all]{xy} 
\numberwithin{equation}{section}

\newtheorem{teo}{Theorem}[section] 
\newtheorem{lem}[teo]{Lemma}

\newtheorem{cor}[teo]{Corollary}
\newtheorem{defi}[teo]{Definition}
\newtheorem{oss}[teo]{Remark}

\newtheorem{prop}[teo]{Proposition}

\newtheorem{ebsk}[teo]{EbS(k)}

\title{Inductive approach to effective b-semiampleness}
\date{\today}

\author{Enrica Floris}

\address{Enrica Floris\\
IRMA, Universit\'e de Strasbourg et CNRS\\
7 rue Ren\'e-Descartes\\
67084 Strasbourg Cedex\\
France}
\email{floris@math.unistra.fr}

\begin{document}
\begin{abstract}
An lc-trivial fibration is the data of a pair $(X,B)$
and a fibration such that $(K_X+B)|_F$ is torsion where $F$
is the general fibre.
For such a fibration we have an equality
$K_X+B+1/r(\varphi)=f^{\ast}(K_Z+B_Z+M_Z)$
where $B_Z$ is the discriminant of $f$, $M_Z$ is the moduli part and $\varphi$ is a rational function.
It has been conjectured by Prokhorov and Shokurov
that there exists an integer $m=m(r,\dim F)$
such that $mM_{Z'}$ is base-point-free on some birational model $Z'$ of $Z$.
In this work we reduce this conjecture to the case where the base $Z$ has dimension one.
Moreover in the case where $M_Z\equiv 0$ we prove the existence of an integer $m$, that depends only on
the middle Betti number of a canonical covering of the fibre, such that $mM_Z\sim 0$.
\end{abstract}

\maketitle
\section{Introduction}
The canonical bundle formula is a tool for studying
the properties of the canonical bundle of a variety $X$,
such that there exists a fibration $f\colon X\rightarrow Z$,
in terms of the properties of the canonical bundle of $Z$,
of the singularities of the fibration and of the birational variation of the fibres of $f$.

More precisely we consider a pair $(X,B)$ such that there exists
an lc-trivial fibration $f\colon (X,B)\rightarrow Z$,
that is a fibration such that $(K_X+B)|_F$ is a torsion divisor,
where $F$ is a general fibre
(see Section 2 for a complete definition).
Then we can write $$K_X+B+\frac{1}{r}(\varphi)=f^{\ast}(K_Z+B_Z+M_Z)$$
where $\varphi$ is a rational function and $r$ is the minimum integer
such that $r(K_X+B)|_F\sim 0$.
The divisor $B_Z$ is called the \textit{discriminant} and it is defined in terms of some log-canonical thresholds
with respect to the pair $(X,B)$.
Precisely we have
$B_Z=\sum(1-\gamma_p)p$ where
$$\gamma_p=\sup\{t\in\mathbb{R}|\,(X,B+tf^{\ast}(p)) {\rm \:is\: lc\: over\:} p\}.$$
The divisor $M_Z$, called the \textit{moduli part},
is a $\mathbb{Q}$-Cartier divisor and it is nef 
on some birational modification of $Z$
by ~\cite[Theorem 0.2]{Amb04}.
The moduli part should be related to
the birational variation of the fibres of $f$.\\
This is true for instance in the case of elliptic fibrations.
Indeed, if $f\colon X\rightarrow Z$ is an elliptic fibration,
we have Kodaira's canonical bundle formula
$$K_X=f^{\ast}(K_Z+B_Z+M_Z)$$
and
$$12M_Z\sim j^{\ast}\mathcal{O}_{\mathbb{P}^1}(1)$$
where $j\colon Z\rightarrow \mathbb{P}^1=\mathcal{M}_1$ is the application induced by $f$
to the moduli space of elliptic curves.
In particular $M_Z$ is semiample.

In ~\cite{ProkShok} Prokhorov and Shokurov state the following conjecture
(in our statement we specify the dimension of the base).
\begin{ebsk}[Effective b-Semiampleness, Conjecture 7.13.3, ~\cite{ProkShok}]\label{effbsemi}
There exists an integer number $m=m(d,r)$
such that for any lc-trivial fibration
$f\colon (X,B)\rightarrow Z$ with dimension of the generic fibre $F$ equal to $d$,
dimension of $Z$ equal to $k$
and Cartier index of $(F,B|_F)$ equal to $r$
there exists a birational morphism $\nu\colon Z'\rightarrow Z$
such that $mM_{Z'}$ is base point free. 
\end{ebsk}
The relevance of the above conjecture is well illustrated for instance
by a result due to X. Jiang, who proved recently in ~\cite{J1}
that Conjecture \textbf{EbS} \ref{effbsemi} implies a uniformity statement
for the Iitaka fibration of \textit{any} variety of positive Kodaira
dimension under the assumption that the fibres have a good minimal model.\\
Moreover Todorov and Todorov-Xu
prove some unconditional uniformity results for the Iitaka fibration
of varieties of Kodaira dimension at most 2 and Kodaira codimension
1 (Todorov ~\cite[Theorem 1.2]{Tod})
and 2 (Todorov and Xu ~\cite[Theorem 1.2]{TodXu}).
Their proofs relie on the existence of a bound on the denominators of the moduli part
and, for the case of Kodaira codimension one, on the result by Prokhorov and Shokurov
that proved
conjecture \textbf{EbS} for all $k$
in the case where the fibres are curves (see ~\cite[Theorem 8.1]{ProkShok}).\\
The semiampleness of the moduli part has been also proved for all $k$
if $F$ is isomorphic to a K3 surface or to an abelian variety
by Fujino in ~\cite[Theorem 1.2]{Fuji}.\\
It is worth noticing that the proofs of semiampleness in these cases
use the existence of a moduli space for the fibres.\\

The main goal of this work is to develop an inductive approach to the conjecture \textbf{EbS}.
Our first result is the following.
\begin{teo}\label{semicurvesemitutto}
\textbf{EbS(1)} implies \textbf{EbS(k)}.
\end{teo}

An inductive approach on the dimension of the base as in Theorem \ref{semicurvesemitutto}
allows us to prove a result of effective semiampleness
in the case $M_Z\equiv 0$.
Indeed we are able to prove an effective version of ~\cite[Theorem 3.5]{Amb05}.
\begin{teo}\label{mainteo}
There exists an integer number $m=m(b)$
such that for any klt-trivial fibration
$f\colon (X,B)\rightarrow Z$ with
\begin{itemize}
\item $M_Z\equiv 0$;
\item $Betti_{\dim E'}(E')=b$
where $E'$ is a non-singular model of the cover $E\rightarrow F$ 
associated to the unique element of $|r(K_F+B|_F) |$;
\end{itemize}
we have $mM_Z\sim\mathcal{O}_Z$.
\end{teo}

Moreover for the case where the pair $(X,B)$ is lc but not klt 
on the generic point of the base we have the following
\begin{teo}\label{mainteolc}
Let $f\colon (X,B)\rightarrow Z$ be an lc-trivial fibration
 with $M_Z\equiv 0$. Then $M_Z$ is torsion.
\end{teo}

In section 2 we recall some general definitions and results
concerning the canonical bundle formula.
Section 3 is devoted to the proof of Theorem \ref{semicurvesemitutto}.
Section 4 and section 5 contain several results that will be useful in the proofs
of Theorem \ref{mainteo} and Theorem \ref{mainteolc} that are the subject of section 6.

The techniques that we use in the proof of Theorem \ref{mainteo}
come from the theory of variations of Hodge structures as in ~\cite[Theorem 3.5]{Amb05}.
The integer $m(b)$ is determined in the unipotent case by using the semisimplicity theorem of Deligne
~\cite[Theorem 4.2.6]{DeligneII}.
Then we show that the same integer works in the general case.

The proof of Theorem \ref{mainteolc} consists in adaptating
the proof of Theorem \ref{mainteo} to the more general setting of variation of mixed Hodge structures.

\bigskip

\thanks{{\bfseries Acknowledgements.}
I would like to express my gratitude to my Ph.D advisor, Gianluca Pacienza,
for bringing my attention to this problem and for his generous help.
This work owes a great deal to his influence. 
}

\section{Notations, definitions and known results}
We will work over $\mathbb{C}$.
In the following $\equiv$, $\sim$ and $\sim_{\mathbb{Q}}$ will respectively indicate
numerical, linear and $\mathbb{Q}$-linear equivalence of divisors.
The following definitions are taken from ~\cite{KM}.
\begin{defi}
A pair $(X,B)$ is the data of a normal variety $X$ and a $\mathbb{Q}$-Weil divisor $B$ 
such that $K_X+B$ is $\mathbb{Q}$-Cartier.
\end{defi}

\begin{defi}
Let $(X,B)$ be a pair and write $B=\sum b_i B_i$.
Let $\nu\colon Y\rightarrow X$ be a birational morphism, $Y$ normal.
We can write
$$K_Y\equiv \nu^{\ast}(K_X+B)+\sum a(E_i,X,B) E_i.$$
where $E_i\subseteq Y$ are distinct prime divisors and $a(E_i,X,B)\in\mathbb{R}$.
Furthermore we adopt the convention that a nonexceptional divisor $E$ appears in the sum
if and only if $E=\nu_{\ast}^{-1}B_i$ for some $i$
and then with coefficient $a(E,X,B)=-b_i$.\\
The $a(E_i,X,B)$ are called discrepancies.

A divisor $E$ is exceptional over $X$ if there exists a birational morphism $\nu\colon Y\rightarrow X$
and $E\subseteq Y$ and it is exceptional for $\nu$.
\end{defi}

\begin{defi}
Let $(X,B)$ be a pair and $f\colon X\rightarrow Z$ be a morphism.
Let $o\in Z$ be a point.
A log resolution of $(X,B)$ over $o$
is a birational morphism $\nu\colon X'\rightarrow X$
such that for all $x\in f^{-1}o$
the divisor $\nu^{\ast}(K_X+B)$ is
simple normal crossing at $x$.
\end{defi}

\begin{defi}
We set
$$\rm{discrep}(X,B)
=\inf \{a(E,X,B)\;|\; E\, exceptional\, divisor\, over\, X\}.$$
A pair $(X,B)$ is defined to be
\begin{itemize}
\item klt (kawamata log terminal) if $\rm{discrep}(X,B)>-1$,
\item lc (log canonical) if $\rm{discrep}(X,B) \geq -1.$
\end{itemize}
\end{defi}

\begin{defi}
Let $f\colon (X,B)\rightarrow Z$ be a morphism and $o\in Z$ a point.
For an exceptional divisor $E$ over $X$ we set
$c(E)$ its image in $X$.
We set $$\rm{discrep}_o(X,B)
=\inf \{a(E,X,B)\;|\; E\, {\rm exceptional}\, {\rm divisor}\, {\rm over}\, X,\; f(c(E))=o\}.$$
A pair $(X,B)$ is defined to be
\begin{itemize}
\item klt over $o$ (kawamata log terminal) if $\rm{discrep}_o(X,B)>-1$,
\item lc over $o$ (log canonical) if $\rm{discrep}_o(X,B) \geq -1.$
\end{itemize}
\end{defi}

\begin{defi}
Let $(X,B)$ be a pair.
A place for $(X,B)$ is a prime divisor on some birational model $\nu \colon Y\rightarrow X$ of $X$
such that $a(E,X,B)=-1$.
The image of $E$ in $X$ is called a centre.
\end{defi}

\begin{defi}
Let $(X,B)$ be a pair and $\nu\colon X'\rightarrow X$
a log resolution of the pair.
We set $$A(X,B)=K_{X'}-\nu^{\ast}(K_X+B)$$
and $$A^{\ast}(X,B)=A(X,B)+\sum_{a(E,X,B)=1} E.$$
\end{defi}

\begin{defi}
A klt-trivial (resp. lc-trivial) fibration $f \colon (X,B) \rightarrow Z$ consists of a contraction
of normal varieties $f \colon X \rightarrow Z$ and of a log pair $(X,B)$ satisfying
the following properties:
\begin{enumerate}
\item $(X,B)$ has klt (resp. lc) singularities over the generic
point of $Z$;
\item ${\rm rank}\, f'_{\ast}\,\mathcal{O}_X(\lceil A(X,B)\rceil) = 1$
(resp. ${\rm rank}\, f'_{\ast}\,\mathcal{O}_X(\lceil A^{\ast}(X,B)\rceil) = 1$)
where $f'=f\circ\nu$ and $\nu$ is a given log resolution of the pair $(X,B)$;
\item there exists a positive integer $r$, a rational function $\varphi\in k(X)$
and a $\mathbb{Q}$-Cartier divisor $D$ on $Z$ such that
$$K_X + B +\frac{1}{r}(\varphi) = f^{\ast}D.$$
\end{enumerate}
\end{defi}

\begin{oss}
{\rm
Condition (2) is verified for instance if $B$ is effective because
$$\lceil A^{\ast}(X,B)\rceil=\lceil K_{X'}-\nu^{\ast}(K_X+B)+\sum_{a(E,X,B)=1} E\rceil$$
is exceptional.
}
\end{oss}
\begin{oss}\label{cartindex}
{\rm
The smallest possible $r$ is the minimum of the set
$$\{m\in\mathbb{N}|\,m(K_X+B)|_F\sim 0\}$$
that is the Cartier index of the fibre.
We will always assume that the $r$ that appears in the formula is the smallest one.
}
\end{oss}

\begin{defi}
Let $p\subseteq Z$ be a codimension one point.
The log canonical threshold of $f^{\ast}(p)$ with respect to the pair $(X,B)$ is
$$\gamma_p=\sup\{t\in\mathbb{R}|\,(X,B+tf^{\ast}(p)) {\rm \:is\: lc\: over\:} p\}.$$
We define the {\rm discriminant} of $f \colon (X,B) \rightarrow Z$ as
\begin{eqnarray}\label{discriminant}
B_Z&=&\sum_{p}(1-\gamma_p)p.
\end{eqnarray}
\end{defi}
We remark that, since the above sum is finite, $B_Z$ is a $\mathbb{Q}$-Weil divisor.

\begin{defi}
Fix $\varphi\in\mathbb{C}(X)$ such that $K_X + B +\frac{1}{r}(\varphi) = f^{\ast}D$.
Then there exists a unique divisor $M_Z$ such that we have
\begin{eqnarray}\label{cbf}
K_X + B +\frac{1}{r}(\varphi) &=& f^{\ast}(K_Z+B_Z+M_Z)
\end{eqnarray}
where $B_Z$ is as in (\ref{discriminant}).
The $\mathbb{Q}$-Weil divisor $M_Z$ is called the {\rm moduli part}.
\end{defi}
We have the two following results.\\
\begin{teo}[Theorem 0.2 ~\cite{Amb04}, ~\cite{Corti}]\label{nefness}
Let $f \colon(X,B)\rightarrow Z$ be an lc-trivial fibration.
Then there exists a proper birational morphism
$Z'\rightarrow Z$ with the following properties:
\begin{description}
\item[(i)] $K_{Z'}+B_{Z'}$ is a $\mathbb{Q}$-Cartier divisor, 
and $\nu^{\ast}(K_{Z'}+B_{Z'}) = K_{Z''}+B_{Z''}$
for every proper birational morphism $\nu \colon Z''\rightarrow Z'$.
\item[(ii)] $M_{Z'}$ is a nef $\mathbb{Q}$-Cartier divisor and $\nu^{\ast}(M_{Z'}) = M_{Z''}$ for every
proper birational morphism $\nu \colon Z''\rightarrow Z'$.
\end{description}
\end{teo}
\begin{prop}[Proposition 5.5 ~\cite{Amb04}]\label{basechange}
Let $f\colon (X,B)\rightarrow Z$ be an lc-trivial fibration.
Let $\tau\colon Z'\rightarrow Z$ be a generically finite projective
morphism from a non-singular variety $Z'$. Assume there exists a simple
normal crossing divisor $\Sigma_{Z'}$ on $Z'$ which contains $\tau^{-1}\Sigma_Z$ and the
locus where $\tau$ is not \'etale. Let $M_{Z'}$ be the moduli part of the induced
lc-trivial fibration $f'\colon (X',B')\rightarrow Z'$.
Then $M_{Z'}=\tau^{\ast}M_Z$.
\end{prop}
\begin{teo}[Inverse of adjunction, Proposition 3.4, ~\cite{Amb99}, see also Theorem 4.5 ~\cite{FMor}] Let $f\colon(X,B)\rightarrow Z$ be an lc-trivial
fibration. 
Then $(Z,B_Z )$ has klt (lc)
singularities in a neighborhood of a point $p\in Z$ if and only if $(X,B)$ has
klt (lc) singularities in a neighborhood of $f^{-1}p$.
\end{teo}

The Formula (\ref{cbf}) is called \textit{canonical bundle formula}.

\section{Reduction theorems}

Throughout this part we will assume that the bases of the lc-trivial fibrations are smooth varieties.
\begin{lem}\label{modulirestr}
Let $f\colon (X,B)\rightarrow Z$ be an lc-trivial fibration.
Then there exists a hyperplane section $H\subseteq Z$
such that $M_Z|_H=M_H$.
\end{lem}

\begin{proof}
The set $$\mathcal{S}=\left\{
\begin{array}{l}
\scriptstyle 
o\; {\rm point\; of}\; Z\;{\rm of\; codimension\; 1\; such\; that\;
the\; log\; canonical} \\
\scriptstyle {\rm threshold\; of}\; f^{\ast}o\;
{\rm with\; respect\; to}\; (X,B)\; {\rm is\; different\; from}\; 1
\end{array} 
\right\}$$
is a finite set.\\
By the Bertini theorem,
since $Z$ is smooth,
we can find a hyperplane section $H\subseteq Z$ such that
\begin{enumerate}
\item $H$ is smooth;
\item $H$ intersects $o$ transversally and generically for every $o\in\mathcal{S}$;
\item $H$ does not contain any intersection $o\cap o'$
where $o'\in\mathcal{S}\backslash\{o\}$.
\end{enumerate}
The restriction $f_H\colon X\cap f^{-1}(H)\rightarrow H$
is again an lc-trivial fibration.
Set $$X_H=f^{-1}(H); \;\;B_{X_H}=B|_{X_H};\;\; o_H=o\cap H.$$
The canonical bundle formula for $f_H$ is
$$K_{X_H}+B_{X_H}+\frac{1}{r}(\psi)=f_H^{\ast}(K_H+B_H+M_H).$$
Then the log canonical threshold of $f_H^{\ast}o_H$ 
with respect to $(X_H,B_H)$ is equal to
the log canonical threshold of $f^{\ast}o$ 
with respect to $(X,B)$.
We have then $B_Z|_H=B_H$.

If we write the canonical bundle formula for $f$, we have
$$K_X+B+\frac{1}{r}(\varphi)=f^{\ast}(K_Z+B_Z+M_Z).$$
If we sum $f^{\ast}H$ on both sides of the equality, restrict to $f^{-1}H=X_H$ and
apply the adjunction formula, we obtain
$$K_{X_H}+B_{X_H}+\frac{1}{r}(\varphi|_{X_H})=f_H^{\ast}(K_H+B_Z|_H+M_Z|_H).$$
Since we have $B_Z|_H=B_H$, we must also have $M_Z|_H=M_H$.
\end{proof}
Lemma \ref{modulirestr} is the main tool in order to prove by induction Theorem \ref{semicurvesemitutto}.

\begin{prop}\label{codim}
Conjecture \textbf{EbS(1)} implies that for all
lc-trivial fibration $f\colon X\rightarrow Z$ we have 
$${\rm codim}({\rm Bs}|mM_Z|)\geq 2$$
where $m$ is as in Conjecture \textbf{EbS(1)}.
\end{prop}
\begin{proof}
We prove the statement by induction on $k=\dim Z$.
The base of induction is $\dim Z=1$ and this case follows from \textbf{EbS(1)}.
Suppose then that the statement is true for lc-trivial fibration whose base hase dimension $k-1$
and let $f\colon X\rightarrow Z$ be an lc-trivial fibration with $\dim Z=k>1$.
Then we have $$|mM_Z|=|M|+{\rm Fix}$$ where ${\rm Fix}$ is the fixed part of the linear system and
${\rm codim}({\rm Bs}|M|)\geq 2$.
Let $H$ be a hyperplane section as in Lemma \ref{modulirestr},
such that $H-mM_Z$ is ample.
By the Kodaira vanishing theorem $$h^0(Z,mM_Z-H)=h^1(Z,mM_Z-H)=0$$
and the restriction induces an isomorphism
$$H^0(Z,mM_Z)\cong H^0(H,mM_Z|_H)\cong H^0(H,mM_H).$$
Then if we write 
$$
\begin{array}{rcl}
|mM_Z|_{|H}&=&|M|_{|H}+{\rm Fix}|_H\\
|mM_H|&=&|L|+{\rm fix}
\end{array}
$$
where ${\rm fix}$ is the fixed component
of the linear system $|mM_H|$, we have ${\rm fix} \supseteq {\rm Fix}|_H$.
And since by inductive hypothesis ${\rm fix}=0$ also ${\rm Fix}|_H=0$ and then ${\rm Fix}=0$.
\end{proof}
\begin{cor}\label{duesez}
Conjecture \textbf{EbS(1)} implies that for any
lc-trivial fibration $f\colon X\rightarrow Z$ we have $h^0(Z,mM_Z)\geq2$,
unless $M_Z$ is torsion, where $m$ is as in \textbf{EbS(1)}.
\end{cor}
\begin{proof}
By Proposition \ref{codim} there must be at least two sections,
unless $M_Z$ is torsion.
\end{proof}

\begin{proof} [Proof of Theorem \ref{semicurvesemitutto}]
\underline{We treat first the torsion case},
we prove by induction on the dimension of the base of the lc-trivial fibration
that there exists an integer $m=m(d,r)$ such that $mM_Z\cong\mathcal{O}_Z$.
If the dimension of the base equals one then it follows from Conjecture \textbf{EbS(1)}.
Assume then that $f\colon X\rightarrow Z$ is an lc-trivial fibration with $\dim Z=k>1$
and $M_Z$ is torsion, that is there exists an integer $a$ such that $aM_Z\cong \mathcal{O}_Z$.

Let $H$ be a hyperplane section such that $M_Z|_H=M_H$, as in Lemma \ref{modulirestr},
and such that $H-mM_Z$ is an ample divisor.
Since $M_Z|_H=M_H$, also $M_H$ is torsion because
$$\mathcal{O}_H\cong\mathcal{O}_Z|_H\cong aM_Z|_H\cong aM_H.$$
By the Kodaira vanishing theorem, since $k>1$ and $H-mM_Z$ is an ample divisor, we have
$$H^0(Z,mM_Z)\cong H^0(H,mM_H).$$
By the inductive hypothesis $mM_H$ is base-point-free, then $h^0(H,mM_H)=1$.
Thus also $h^0(Z,mM_Z)=1$ and $mM_Z\cong \mathcal{O}_Z$.\\
\underline{Then we assume that $M_Z$ is not torsion} and
we prove the statement by induction on the dimension of the base of the lc-trivial fibration.
The one-dimensional case is exactly Conjecture \textbf{EbS(1)}.
Suppose then that the statement is true for all the lc-trivial fibrations whose base has dimension
$k-1$ and let $f\colon X\rightarrow Z$ be an lc-trivial fibration
with $\dim Z=k$.
Let $Z' \rightarrow Z$ be the birational model given by 
Theorem \ref{nefness}\textbf{(ii)}.
We prove that $mM_{Z'}$ is base-point-free.
Let $\nu\colon \hat{Z} \rightarrow Z'$ be a resolution of the linear system
$|mM_{Z'}|$. Then $\nu^{\ast}|mM_{Z'}|=|{\rm Mob}|+{\rm Fix}$
where $|{\rm Mob}|$ is a base-point-free linear system and ${\rm Fix}$ is the fixed part.
We have $$\nu^{\ast}|mM_{Z'}|=|\nu^{\ast}(mM_{Z'})|=|mM_{\hat{Z}}|.$$
Since by Proposition \ref{codim} we have ${\rm codim}({\rm Bs}|mM_{\hat{Z}}|)\geq 2$,
it follows that ${\rm Fix}= 0$ and $|mM_{Z'}|$ is base-point-free.
\end{proof}

\begin{oss}\label{effeff}
{\rm
By considering as in Proposition \ref{codim}
the long exact sequence associed to
$$0\rightarrow \mathcal{O}_Z(mM_Z-H)\rightarrow \mathcal{O}_Z(mM_Z)\rightarrow \mathcal{O}_H(mM_Z|_H)\rightarrow0$$
for a hyperplane section $H$ as in Lemma \ref{modulirestr},
it is possible to also prove an inductive result on \textbf{effective non-vanishing}.
That is, the existence of an integer $m=m(d,r)$
such that $H^0(Z,mM_Z)\neq 0$ for all lc-trivial fibrations
$f\colon (X,B)\rightarrow Z$ with $\dim Z=1$
implies the same result for lc-fibrations with $\dim Z=k\geq 1$ 
(and with same dimension of the fibres and Cartier index).
}
\end{oss}

\section{Variation of Hodge structures and covering tricks}
\subsection{Variation of Hodge structures}
Let $\mathcal{S}$ be $\mathbb{C}^{\ast}$ viewed as an $\mathbb{R}$-algebra.
\begin{defi}[2.1.4 ~\cite{DeligneII}]
A real Hodge structure is a real vector space $V$ of finite dimension
together with an action of $\mathcal{S}$.\\
The representation of $\mathcal{S}$ on $V$ induces a bigraduation on $V$,
such that $\overline{V^{pq}}=V^{qp}$.
We say that $V$ has weight $n$ if $V^{pq}= 0$ whenever $p+q\neq n$.
\end{defi}

\begin{defi}[2.1.10 ~\cite{DeligneII}]
A Hodge structure $H$ of weight $n$ is
\begin{itemize}
\item a $\mathbb{Z}$-module of finite type $H_{\mathbb{Z}}$;
\item a real Hodge structure of weight $n$ on $H_{\mathbb{R}}=H_{\mathbb{Z}}\otimes_{\mathbb{Z}}\mathbb{R}$.
\end{itemize}
\end{defi}

\begin{defi}
Let $S$ be a topological space. 
A local system on $S$ is a sheaf $\mathbb{V}$ of $\mathbb{Q}$-vector spaces on $S$.
\end{defi}
Let now $S$ be a complex manifold.
\begin{defi}
Let $\mathcal{V}\rightarrow S$ be a vector bundle.
A connection is a morphism $$\nabla\colon\mathcal{V}\rightarrow\Omega^1_S\otimes\mathcal{V}$$
that satisfies the Leibniz rule.\\
The curvature of a connection is $\nabla\circ\nabla\colon\mathcal{V}\rightarrow\Omega^2_S\otimes\mathcal{V}.$\\
A connection is said to be integrable if $\nabla\circ\nabla=0$.
\end{defi}
By ~\cite[Proposition 2.16]{DeligneEqdiff} the data of a local system $\mathbb{V}$ is equivalent to the data of a vector bundle
$\mathcal{V}\rightarrow S$ 
together with an integrable connection $\nabla$
and the correspondance is given by associating to $\mathbb{V}$
the vector bundle $$\mathcal{V}=\mathbb{V}\otimes \mathcal{O}.$$

\begin{defi} 
A flat subsystem of a local system $\mathbb{V}$
is a sub-local system $\mathbb{W}$ of $\mathbb{V}$
or equivalently a subbundle $\mathcal{W}$ of $\mathcal{V}$
on which the curvature of the connection is zero.
\end{defi}

\begin{defi}[(3.1) ~\cite{SteenZuck}]
A variation of Hodge structure of weight $m$ on $S$ is:
\begin{itemize}
\item a local system $\mathbb{V}$ on $S$;
\item a flat bilinear form $$Q\colon \mathcal{V}\times\mathcal{V}\rightarrow\mathbb{C}$$
which is rational with respect to $\mathbb{V}$, where $\mathcal{V}=\mathbb{V}\otimes\mathcal{O}_S$;
\item a \textit{Hodge filtration} $\{\mathcal{F}^p\}$, that is a decreasing filtration of $\mathcal{V}$
by holomorphic subbundles such that for all $p$ we have $\nabla(\mathcal{F}^p)\subseteq\Omega_S^1\otimes\mathcal{F}^{p-1}.$
\end{itemize}
\end{defi}

\begin{defi}[(3.4) ~\cite{SteenZuck}]
A variation of mixed Hodge structure on $S$ is:
\begin{itemize}
\item a local system $\mathbb{V}$ on $S$;
\item a \textit{Hodge filtration} $\{\mathcal{F}^p\}$ that is a decreasing filtration 
 of $\mathcal{V}$ by holomorphic subbundles
such that for all $p$ we have $\nabla(\mathcal{F}^p)\subseteq\Omega_S^1\otimes\mathcal{F}^{p-1}$;
\item a \textit{Weight filtration} $\{\mathcal{W}_k\}$ that is an increasing filtration of $\mathcal{V}$
by local subsystems, or equivalently, the subsheaf $\mathcal{W}_k$ is defined over $\mathbb{Q}$ for every $k$;
\end{itemize}
Moreover we require that the filtration induced by $\{\mathcal{F}^p\}$
on $\mathcal{W}_k/\mathcal{W}_{k-1}$ determines a variation of Hodge structure of weight $k$.
\end{defi}

From now on we will be interested in variations of Hodge structures and of mixed Hodge structures defined on
a Zariski open subset $Z_0$ of a projective variety $Z$. We assume moreover that $\Sigma_Z=Z\backslash Z_0$
is a simple normal crossing divisor.

The following is a fundamental result about the behaviour of a variation of Hodge structures on $Z_0$
near $\Sigma_Z$.
For the definition of \textit{monodromy} and \textit{unipotent monodromy} 
of variations of Hodge structures and residue of a connection
see ~\cite[Definition 10.16, section 11.1.1]{PetersSteen}
\begin{prop}[Proposition 5.2(d), ~\cite{DeligneEqdiff}]\label{canext}
Let $\mathcal{V}$ be a variation of Hodge structures on $Z_0$
that has unipotent monodromies around $\Sigma_Z$. Let $z$
be a local variable with centre in $\Sigma_Z$. Then
\begin{description}
\item[a] There exists a unique extension $\tilde{\mathcal{V}}$ of $\mathcal{V}$ on $Z$
such that  
\begin{description}
\item[i] every horizontal section of $\mathcal{V}$ as a section of $\tilde{\mathcal{V}}$ on $Z_0$
 grows at most as $$O(\log \parallel z\parallel^k)$$ near $\Sigma_Z$;
\item[ii] let $\mathcal{V}^{\ast}$ be the dual of $\mathcal{V}$.
Every horizontal section of $\mathcal{V}^{\ast}$
 grows at most as $$O(\log \parallel z\parallel^k)$$ near $\Sigma_Z$.
\end{description}
\item[b] Conditions (i) and (ii) are equivalent respectively to conditions (iii) and (iv).
\begin{description}
\item[iii] The  matrix of the connection on $\mathcal{V}$
on a local frame for $\tilde{\mathcal{V}}$
has logarithmic poles near $\Sigma_Z$.
\item[iv] Each residue of the connection along each irreducible component of $\Sigma_Z$
is nilpotent.
\end{description}
\item[c] Let $\mathcal{V}_1$ and $\mathcal{V}_2$ be variations of Hodge structures on $Z_0$
that has unipotent monodromies around $\Sigma_Z$.
Every morphism $f\colon\mathcal{V}_1\rightarrow\mathcal{V}_2$
extends to a morphism $\tilde{\mathcal{V}}_1\rightarrow\tilde{\mathcal{V}}_2$.
Moreover the functor $\mathcal{V}\mapsto\tilde{\mathcal{V}}$ is exact and
commutes with $\otimes$, $\wedge$, ${\rm Hom}$.
\end{description}
The extension $\tilde{\mathcal{V}}$ is called \textbf{canonical extension}.
\end{prop}
\begin{oss}\label{lamatrdellaconn}
{\rm
In the situation of Proposition \ref{canext}
the matrix of the connection on $\mathcal{V}$
has the following form
$$\Gamma=\sum U_i \frac{d z_i}{z_i}$$
where $U_i$ is the matrix that represents the nilpotent part of the monodromy around the
component $\Sigma_i$ of $\Sigma_Z$.
}
\end{oss}
Let $V$, $Z$ be nonsingular projective varieties and $h\colon V\rightarrow Z$
a surjective morphism with connected fibres.
Let $Z_0\subseteq Z$ be the largest Zariski open set where $h$ is smooth
and $V_0=h^{-1}Z_0$.
Assume that $\Sigma_Z=Z\backslash Z_0$ and $\Sigma_V=V\backslash V_0$ are simple normal crossing divisors
on $Z$ and $V$. Set $d=\dim V-\dim Z$.\\
Consider $\mathcal{H}_{\mathbb{C}}=(R^d h_{\ast}\mathbb{C}_{V_0})_{prim}$
and $\mathcal{H}_0=\mathcal{H}_{\mathbb{C}}\otimes\mathcal{O}_{Z_0}$
and set $\mathcal{F}=h_{\ast}\omega_{V/Z}$ and $\mathcal{F}_0=\mathcal{F}\otimes\mathcal{O}_{Z_0}$.
We have that $\mathcal{H}_{\mathbb{C}}$ is a local system over $Z_0$.
Moreover $\mathcal{H}_0$ has a descending filtration $\{\mathcal{F}^p\}_{0\leq p\leq d}$,
the \textit{Hodge filtration} and $\mathcal{F}_0=\mathcal{F}^d$.\\
There is a canonical way to extend $\mathcal{H}_0$ and $\mathcal{F}_0$ on $Z$:
\begin{teo}[Proposition 5.4 ~\cite{DeligneEqdiff}, Theorem 2.6 ~\cite{Kol86}]
\begin{enumerate}
\item $R^d h_{\ast}\mathbb{C}_{V_0}\otimes\mathcal{O}_{Z_0}$
has a canonical extension to a locally free sheaf on $Z$.
\item $h_{\ast}\omega_{V/Z}$ coincides with the canonical extension
of the bottom piece of the Hodge filtration.
\end{enumerate}
\end{teo}

Let $h_0\colon V_0\rightarrow Z_0$ be as before.
Let $D\subseteq V$ be a simple normal crossing divisor such that
the restriction $h_0|_{D}$ is flat.
Assume that $D+\Sigma_V$ is simple normal crossing.
Let us denote the restriction as $$h'_0\colon V_0\backslash D\rightarrow Z_0.$$
Thus $R^d (h'_0)_{\ast}\mathbb{C}_{V_0\backslash D}$ is a local system on $Z_0$
by ~\cite[section 5.2]{SteenZuck}.
Let $\{\mathcal{F}^p\}$ be the Hodge filtration
and let $$\mathcal{W}_k
=(h_0)_{\ast}\Omega^k_{V_0/Z_0}(\log D)\otimes\Omega_{V_0/Z_0}^{\bullet-k}
=\{(h_0)_{\ast}\Omega^k_{V_0/Z_0}(\log D)\otimes\Omega_{V_0/Z_0}^{s-k}\}_s$$
be the weight filtration of the complex $(h_0)_{\ast}\Omega^{\bullet}_{V_0/Z_0}(\log D)$.
In particular $\mathcal{W}_k$ is a complex.
When we will need to refer to the trace left by $\mathcal{W}_k$
on an element $(h_0)_{\ast}\Omega^s_{V_0/Z_0}(\log D)$ of the complex
we will write $\mathcal{W}_k((h_0)_{\ast}\Omega^s_{V_0/Z_0}(\log D))$.

Let $h\colon V\rightarrow Z$ be a morphism such that $\Sigma_Z$
is a simple normal crossing divisor
and let $\tau\colon Z'\rightarrow Z$ be a morphism
from a nonsingular variety $Z'$ such that $\tau^{-1}\Sigma_Z$
is a simple normal crossing divisor.
Let $V'$ be a desingularization of the component of $V\times_Z Z'$
that dominates $Z'$.
$$
\xymatrix{
V' \ar[d]_{h'} \ar[r] 
& V \ar[d]^{h}\\
Z' \ar[r]_{\tau}
& Z. }
$$
Assume now that $h'$ and $h$ are such
that $R^d h_{\ast}\mathbb{C}_{V_0}$ and $R^d h'_{\ast}\mathbb{C}_{V'_0}$ have unipotent monodromies.
By Proposition \ref{canext}[c] we have a commutative diagram
of sheaves on $Z'$
$$
\xymatrix{
(h')_{\ast}\omega_{V'/Z'} \ar[d]_{\alpha} \ar[r]^{\sim} & (i')_{\ast}(\mathcal{F'}_{0})\cap \mathcal{H}' \ar[d]^{\beta}\\
\tau^{\ast}(h)_{\ast}\omega_{V/Z} \ar[r]^{\sim} & \tau^{\ast}(i)_{\ast}(\mathcal{F}_{0})\cap \mathcal{H}. }
$$
where $i\colon Z_{0}\rightarrow Z$, $i'\colon Z'_{0}\rightarrow Z'$ are inclusions,
$\mathcal{H}$ (resp. $\mathcal{H'}$) is the canonical extension of $R^d h_{\ast}\mathbb{C}_{V_{0}}$
(resp. $R^d h'_{\ast}\mathbb{C}_{V'_{0}}$),
$\mathcal{F}_0=h_{\ast}\omega_{V/Z|_{Z_0}}$ (resp. $\mathcal{F'}_0=h'_{\ast}\omega_{V'/Z'|_{Z'_0}}$)
and $\alpha,\beta$ are the pullbacks by $\tau$.
If we have an isomorphism $$\alpha\colon (h')_{\ast}\omega_{V'/Z'}\rightarrow \tau^{\ast}(h)_{\ast}\omega_{V/Z}$$
then for all $p\in Z'$ we have an isomorphism of $\mathbb{C}$-vector spaces
$$\alpha_p \colon ((h')_{\ast}\omega_{V'/Z'})_p\rightarrow (h_{\ast}\omega_{V/Z})_{\tau(p)}.$$
If $\tau$ is a birational automorphism of $Z$ that fixes $p$,
then $\alpha_p$ is an element of the linear group of $((h)_{\ast}\omega_{V/Z})_p$.
In particular we have the following:

\begin{prop}\label{azioneazione}
Let $h\colon V\rightarrow Z$ be a fibration 
such that $R^d h_{\ast}\mathbb{C}_{V_0}$ has unipotent monodromies.
Assume that we have an action of a group $G$ on $Z$
given by a homomorphism $$G\rightarrow {\rm Bir}(Z)=\{\nu\colon Z\dasharrow Z|\,\nu\,{\rm is}\,{\rm birational}\}.$$
Let $G_p$ be the stabilizer of $p\in Z$.
Then we have an induced action of $G_p$ on $\mathcal{H}_p$ and on $(h_{\ast}\omega_{V/Z})_p$
and these actions commute with the inclusion $(h_{\ast}\omega_{V/Z})_p\subseteq \mathcal{H}_p$.
\end{prop}

By ~\cite[Theorem 17]{Kaw81} we can assume that
$R^d h_{\ast}\mathbb{C}_{V_0}$
(or, more in general, a local system that has quasi-unipotent monodromies) has unipotent monodromies
modulo a finite base change by a Galois morphism.
We restate Kawamata's result in a more precise way that is useful for our purposes.
\begin{teo}[Theorem 17, Corollary 18 ~\cite{Kaw81}]\label{kawamio}
Let $h\colon V\rightarrow Z$ be an algebraic fibre space.
Let $Z_0\subseteq Z$ be the largest Zariski open set where $h$ is smooth
and $V_0=h^{-1}Z_0$.
Assume that $\Sigma_Z=Z\backslash Z_0$ and $\Sigma_V=V\backslash V_0$ are simple normal crossing divisors
on $Z$ and $V$. Set $d=\dim V-\dim Z$.\\
Then there exists a finite surjective morphism $\tau\colon Z'\rightarrow Z$
from a nonsingular projective algebraic variety $Z'$ such that for a desingularization
$V'$ of $V\times_{Z} Z'$ the morphism $h'\colon V'\rightarrow Z'$
induced from $h$ is such that
$R^d h'_{\ast}\mathbb{C}_{V'_0}$ has unipotent monodromies.

Moreover $\tau$ is a composition of cyclic coverings $\tau_j$
$$\tau\colon Z'=Z_{k+1}\xrightarrow{\tau_k} Z_k\ldots Z_2\xrightarrow{\tau_1} Z_1=Z$$
where $\tau_j$ is defined by the building data
$$A_j^{\otimes \delta_j}\sim H_j$$
where $A_j$ is very ample on $Z_j$ and $H_j$ is simple normal crossing.
\end{teo}

We have the following results by Deligne. We state them in our situation, but they hold in a more general setting.
\begin{teo}[Theorem 4.2.6 ~\cite{DeligneII}]\label{semisimple}
The representation of the fundamental group $\pi_1(Z_0,z)$ on the fibre $\mathcal{H}_{\mathbb{Q},z}$
is semi-simple.
\end{teo}
Here we prove a slight modification of ~\cite[Corollary 4.2.8(ii)(b)]{DeligneII}.
\begin{cor}\label{gabolomixed}
Let $W$ be a local subsystem of $\mathcal{V}_{\mathbb{C}}$ of rank one.
Let $b=\dim \mathcal{V}$.
Let $$m(x)={\rm lcm}\{k|\;\phi(k)\leq x\}$$ where $\phi$ is the Euler function.
Then $W^{\otimes m(b)}$ is a trivial local system. 
\end{cor}
\begin{proof}
By Theorem \ref{semisimple} we can write
$$\mathcal{H}_{\mathbb{C},z}=\bigoplus_{i=1}^r \mathcal{H}_i$$
where the $\mathcal{H}_i$ are the isotypic compontents (that is the components that are direct sum
of simple representations of the same weight).
Since $\dim W=1$, the subspace $W_z$ is contained in one isotypic component, let's say $\mathcal{H}_1$.
We have $$\mathcal{H}_1=\mathcal{H}_{\lambda}^{\oplus k},$$
where $\mathcal{H}_{\lambda}$ is a simple component of weight $\lambda$.
Since $W_z$ is simple, it identifies to one of the $\mathcal{H}_{\lambda}$'s and then
$$\bigwedge^k \mathcal{H}_1\cong W_z^{\otimes k}.$$
If $\chi$ is the character that determines $W_z$ as a representation
then the character $\chi^k$ determines $\bigwedge^k \mathcal{H}_1$.
Let $\mathcal{S}$ be the real algebraic group $\mathbb{C}^{\ast}$.
We have an action of $\mathcal{S}$ on $\mathcal{H}_1$.
Indeed by ~\cite[Corollary 4.2.8(ii)(a)]{DeligneII} for all $t\in\mathcal{S}$ we have
$t\mathcal{H}_1\cong\mathcal{H}_1$. In particular $t\mathcal{H}_1$ and $\mathcal{H}_1$ have the same weight.
But since $\mathcal{H}_1$ is isotypic, we have $t\mathcal{H}_1=\mathcal{H}_1$.
The vector space $\hat{\mathcal{H}}_1=\mathcal{H}_1+\bar{\mathcal{H}_1}$
is real and $\mathcal{S}$-invariant,
thus it is defined by a real Hodge substructure of $\mathcal{H}_{\mathbb{R},z}$.
Thus a polarization on $\mathcal{H}$ induces a non-degenerate bilinear form on $\hat{\mathcal{H}}_1$
that is invariant under the action of $\pi_1(Z_0,z)$.\\
Let $e=\dim \hat{\mathcal{H}}_1$. The Hodge structure on $\hat{\mathcal{H}}_1$
comes from an integer structure and thus $(\bigwedge^e \hat{\mathcal{H}}_1)^{\otimes 2}$
is trivial because $\pi_1(Z_0,z)$ can act on a $\mathbb{Z}$-module of rank one only by $\pm 1$.
Since $\hat{\mathcal{H}}_1=\mathcal{H}_1+\bar{\mathcal{H}_1}$, there are two possibilities:
\begin{description}
\item[(a)] if $\mathcal{H}_1$ is real, the character $\chi^{2k}$ is trivial 
\item[(b)] else $\chi^{2k}\bar{\chi}^{2k}$ is trivial.
\end{description}
In any case $|\chi|=1$.\\
The representation of $\pi_1(Z_0,z)$ on $\mathcal{H}_{\mathbb{C},z}$
comes from a representation on $\mathcal{H}_{\mathbb{Q},z}$
and all the conjugated representations of $W_z$ appear in $\mathcal{H}_{\mathbb{C},z}$.
Thus we have at most $b=\dim \mathcal{H}$ conjugated representations of $W_z$.

We proved that for all $\gamma \in\pi_1Z_0,z)$ the number
$\chi(\gamma)$ is a complex number of module one and with at most $b$
complex conjugates.
Thus $\chi(\gamma)$ is a $k$-th root of unity, with $k\leq b$.
If we define $$m(b)={\rm lcm}\{k|\;\phi(k)\leq b\}$$
where $\phi$ is the Euler function,
then $\chi^{m(b)}$ is trivial.
\end{proof}

\subsection{Covering tricks}
In order to give an interpretation of the moduli part
in terms of variation of Hodge structures we need to 
consider an auxiliary log pair $(V,B_V)$
with a fibration $h\colon V\rightarrow Z$.

Let $f\colon X\rightarrow Z$ be an lc-trivial fibration.
Set $\Sigma_Z={\rm Supp}B_Z$ and we assume that $\Sigma_Z$ is a simple normal crossing divisor.
Set $\Sigma_X={\rm Supp}f^{\ast}\Sigma_Z$ and assume that $B+\Sigma_X$
has simple normal crossing support.
We define $g\colon V\rightarrow X$ as the desingularization of the covering induced by the field extension

\begin{eqnarray}\label{cicles}
\mathbb{C}(X)\subseteq\mathbb{C}(X)(\sqrt[r]{\varphi})
\end{eqnarray}

that is, the desingularization of the normalization of $X$ in $\mathbb{C}(X)(\sqrt[r]{\varphi})$
where $\varphi$ is as in (\ref{cbf}).
Let $B_V$ be the divisor defined by the equality $K_V+B_V=g^{\ast}(K_X+B)$.
Set $h=f\circ g\colon V\rightarrow Z$. Then $h$ and $f$ induce the same discriminant and moduli divisor.
Let $\Sigma_V$ be the support of $h^{\ast}\Sigma_Z$ and assume that $\Sigma_V+B_V$
has simple normal crossing support.

The Galois group of (\ref{cicles}) is cyclic of order $r$, then
we have an action of $$\mu_r=\{x\in\mathbb{C}\,|\,x^r=1\}$$
on $g_{\ast}\mathcal{O}_V$.
Then we have also an action of $\mu_r$ on $h_{\ast}\omega_{V/Z}$ and on $h_{\ast}\omega_{V/Z}(P_V)$
where $P_V$ are the horizontal places of the pair $(V,B_V)$.

\begin{prop}[Claim 8.4.5.5, Section 8.10.3 ~\cite{Corti}]
Let $f\colon X\rightarrow Z$ and $V\rightarrow X$ be as above.
The decomposition in eigensheaves is
$$h_{\ast}\omega_{V/Z}=\bigoplus_{i=0}^{r-1}f_{\ast}\mathcal{O}_X(\lceil(1-i)K_{X/Z}-iB+if^{\ast}B_Z+if^{\ast}M_Z\rceil).$$
Let $P_V$ be the places of $(V,B_V)$ and $P$ the places of $(X,B)$. Then we have
$$h_{\ast}\omega_{V/Z}(P_V)=\bigoplus_{i=0}^{r-1}f_{\ast}\mathcal{O}_X(\lceil(1-i)K_{X/Z}-iB+P+if^{\ast}B_Z+if^{\ast}M_Z\rceil)$$
and the righthand-side is the eigensheaf decomposition of the left-hand-side with respect to the action of $\mu_r$.
\end{prop}

\begin{prop}[Proposition 5.2, ~\cite{Amb04}]\label{mautofascio}
Assume that $R^d h_{\ast}\mathbb{C}_{V_0}$ has unipotent monodromies.
Then $M_Z$ is an integral divisor
and it is equal to the eigensheaf $f_{\ast}\mathcal{O}_X(\lceil-B+P+f^{\ast}B_Z+f^{\ast}M_Z\rceil)$
corresponding to a fixed primitive $r$th root of unity.
\end{prop}

\section{Bounding the denominators of the moduli part}

Conjecture \textbf{EbS} \ref{effbsemi} implies in particular the existence
of an integer $N=N(d,r)$ such that for all $f\colon (X,B)\rightarrow Z$
lc-trivial fibration with fibres of dimension $d$
and Cartier index of $(F,B|_F)$ equal to $r$
the divisor $NM_Z$ has integer coefficients.
The result was proved in ~\cite[Theorem 3.2]{Tod} when the fibre is a rational curve.
In ~\cite{EF}, by a different method, we found such integer 
which is considerably smaller than the one in ~\cite{Tod}.
For the reader convenience we present here an argument,
due to Todorov ~\cite[Theorem 3.2]{Tod},
valid in the general case.

\begin{teo}\label{bettieffcart}
There exists an integer number $m=m(b)$
such that for any klt-trivial fibration
$f\colon (X,B)\rightarrow Z$ with
$Betti_{\dim E'}(E')=b$
where $E'$ is a non-singular model of the cover of a general fibre of $f$, $E\rightarrow F$ 
associated to the unique element of $|r(K_F+B|_F) |$
the divisor $mM_Z$ has integer coefficients. 
\end{teo}
We begin by reducing the problem to the case where the base $Z$ is a curve.

\begin{prop}\label{indcart}
If Theorem \ref{bettieffcart} holds for fibrations whose bases have dimension one
then Theorem \ref{bettieffcart} holds for fibrations whose bases have dimension $k\geq 1$.
\end{prop}

\begin{proof}
We prove the statement by induction on $k=\dim Z$.
If $k=1$ then it follows from the hypothesis.
Assume the statement holds for fibrations over bases of dimension $k-1$
and we consider $f\colon X\rightarrow Z$ with $\dim Z=k$.
Let $H$ be a hyperplane section of $Z$ as in Lemma \ref{modulirestr}.
We have thus $M_Z|_H=M_H$. Since $H$ is ample, it meets each component of $M_Z$
and we can choose it such that it meets transversally the components of $M_Z$.
It follows that $NM_Z$ has integer coefficients if and only if $NM_H$ does,
and we are done by inductive hypothesis.
\end{proof}

\begin{proof}[Proof of Theorem \ref{bettieffcart}]
By Proposition \ref{indcart} we can assume that $\dim Z=1$.\\
Consider a finite base change as in Theorem \ref{kawamio}
$$\tau\colon Z'\rightarrow Z.$$
If $h'\colon V'\rightarrow Z'$ is the induced morphism
then $R^d h'_{\ast}\mathbb{C}_{V'_0}$ has unipotent monodromies.
The covering $\tau$ is Galois and let $G$ be its Galois group.
Then we have an action of $G$ on $Z'$
\begin{eqnarray}\label{azaz}
G&\rightarrow& {\rm Bir}(Z')={\rm Aut}(Z').
\end{eqnarray}
By abuse of notation we will denote as $g$ both an element of $G$
and its image in ${\rm Aut}(Z')$.

Let $p'\in Z'$ be a point and
let $e$ be the ramification order of $\tau$ at $p'$.
Let $G_{p'}$ be the stabilizer of $p'$ with respect to the action (\ref{azaz}).
Set $\mu_e=\{x\in\mathbb{C}\,|\,x^e=1\}$.\\
There exists an analytic open set $p'\in U\subseteq Z'$ and a local coordinate $z$ on $U$
centered in $p'$ such that for any $g\in G_{p'}$ there exists $x \in \mu_e$ such that 
$$
\begin{array}{lrcc}
g|_U\colon &U&\longrightarrow &U\\
&z&\longmapsto &xz.
\end{array}
$$
This induces a natural homomorphism $$G_{p'}\rightarrow \mu_e.$$
Then the actions of $G_{p'}$ given by Proposition \ref{azioneazione}
factorize through actions of $\mu_e$:
$$\Phi\colon\mu_e\rightarrow GL((R^d h'_{\ast}\mathbb{C}_{V'})_{p'}),$$
$$\Psi\colon\mu_e\rightarrow GL((h'_{\ast}\omega_{V'/Z'})_{p'})$$
that commute with the inclusion $(h'_{\ast}\omega_{V'/Z'})_{p'}\subseteq (R^d h'_{\ast}\mathbb{C}_{V'})_{p'}$,
that is such that for all $\zeta\in\mu_e$ the restriction of $\Phi(\zeta)$
to $(h'_{\ast}\omega_{V'/Z'})_{p'}$ equals $\Psi(\zeta)$.\\
Thus on $$(h'_{\ast}\omega_{V'/Z'})_{p'}=\bigoplus_{i=0}^{r-1}f_{\ast}\mathcal{O}_X(\lceil(1-i)K_{X/Z}-iB+if^{\ast}B_Z+if^{\ast}M_Z\rceil)$$ we have two actions:
\begin{itemize}
\item one by the group $\mu_e$ that acts on $\varphi$ by a multiplication by an $e$-th rooth of unity,
\item one by the group $\mu_r$ that acts on $\sqrt[r]\varphi$ by a multiplication by an $r$-th rooth of unity.
\end{itemize}
Then there is a $\mu_r\rtimes\mu_e$-action on $(h'_{\ast}\omega_{V'/Z'})_{p'}$
and we can define a $\mu_l$-action on $(h'_{\ast}\omega_{V'/Z'})_{p'}$
where $l=er/(e,r)$.
Since $\mu_r\subseteq\mu_l$ and this second group is commutative,
the action of $\mu_l$ preserves the eigensheaves with respect to the action of $\mu_r$.
By Proposition \ref{mautofascio}, the divisor $M_{Z'}$ is an eigensheaf with respect to the action of $\mu_r$.
Then $\mu_l$ acts on the stalk $\mathcal{O}_{Z'}(M_{Z'})\otimes\mathbb{C}_{p'}$
by a character $\chi_{p'}$.

If for every $p'$ and for every character $\chi_{p'}$
the order of $\chi_{p'}$ divides an integer $N$ then
$$NM_Z=(\tau_{\ast}\mathcal{O}(NM_{Z'}))^G$$
because by Proposition \ref{basechange} we have $M_{Z'}=\tau^{\ast}M_Z$.
Thus $NM_Z$ is a Cartier divisor. 

Let $\mathcal{H}'$ be the canonical extension of the sheaf
$(R^d(h'_0)_{\ast}\mathbb{C})_{prim}\otimes \mathcal{O}_{Z'_0}$ to $Z'$,
where $d$ is the dimension of the fibre of $h'$
and $h'_0\colon V'_0\rightarrow Z'_0$
is the restriction to the smooth locus.
The Hodge filtration also extends and its bottom piece is $h'_{\ast}\mathcal{O}_{V'}(K_{V'/Z'})$.
Then
all the characters that are conjugated to $\chi_{p'}$
must appear as subrepresentations of $\mathcal{H}'_{p'}$ (see also ~\cite[Corollary 4.2.8(ii)(b)]{DeligneII} or
Corollary \ref{gabolomixed}).

If $\chi_{p'}$ acts by a primitive $k$-th root of unity,
then its conjugated subrepresentations are $\phi(k)$
where $\phi$ is the Euler function. 
This bounds $k$ because
then $\phi(l)\leq B_d$, where $B_d=h^d(E',\mathbb{C})$ 
is the $d$-th Betti number.\\
Set $m(x)={\rm lcm}\{ k\,|\,\phi(k)\leq x\}$.
Then $m(B_d)M_Z$ has integer coefficients.
\end{proof}

\section{The case $M_Z\equiv 0$}
\subsection{KLT-trivial fibrations with numerically trivial moduli part.}

The goal of this subsection is the proof of Theorem \ref{mainteo}.
As in Theorem \ref{semicurvesemitutto} the problem can be reduced to the case where the base is a curve.
\begin{prop}\label{trivcurvetrivtuttoklt}
Assume that there exists an integer number $m=m(b)$
such that for any klt-trivial fibration
$f\colon (X,B)\rightarrow Z$ with
\begin{itemize}
\item $\dim Z=1$
\item $M_Z\equiv 0$;
\item $Betti_{\dim E'}(E')=b$
where $E'$ is a non-singular model of the cover $E\rightarrow F$ 
associated to the unique element of $|r(K_F+B|_F) |$;
\end{itemize}
we have $mM_Z\sim\mathcal{O}_Z$.

Then the same holds for bases $Z$ of arbitrary dimension.
\end{prop}
\begin{proof}
We procede by induction on $k=\dim Z$.
The base of induction is the hypothesis of the theorem.\\
Let us assume the statement holds for bases of dimension $k-1$
and prove it for a klt-fibration $f\colon (X,B)\rightarrow Z$
with $\dim Z=k$.
Let $H$ be a hyperplane section, given by Lemma \ref{modulirestr},
such that $M_Z|_H=M_H$.
Let $m$ be the integer given by the inductive hypothesis.
Since $M_Z\equiv 0$ the divisor $H-mM_Z$ is ample.
By taking the long exact sequence associated to
$$0\rightarrow \mathcal{O}_Z(mM_Z-H)\rightarrow \mathcal{O}_Z(mM_Z)\rightarrow \mathcal{O}_H(mM_Z|_H)\rightarrow0$$
we obtain $H^0(Z,mM_Z)\cong H^0(H,mM_H)$
because $H^i(Z,mM_Z-H)=0$ for all $i<\dim Z$.
Then $H^0(Z,mM_Z)\cong\mathbb{C}$, that implies $mM_Z\sim\mathcal{O}_Z$.
\end{proof}

\begin{proof}[Proof of Theorem \ref{mainteo}]
By Proposition \ref{trivcurvetrivtuttoklt}
applied in the case of a klt-trivial fibration,
with $b=h^d(E',\mathbb{C})$,
it is sufficient to prove the statement when the base $Z$ is a curve.

Let us write the canonical bundle formula for $f$:
$$K_X+B+\frac{1}{r}(\varphi)=f^{\ast}(K_Z+B_Z+M_Z).$$
Let $V$ be a nonsingular model of the normalization of $X$ in $\mathbb{C}(X)(\sqrt[r]{\varphi})$.\\
\begin{description}
\item[(i)]\underline{Let us suppose that $R^d h_{\ast}\mathbb{C}_{V_0}$ has unipotent monodromies}.
We argue as in ~\cite{Amb04}, Theorems 4.5 and 0.1.\\
The divisor $M_Z$ is a direct summand in $h_{\ast}\omega_{V/Z}$ (see ~\cite[Lemma 5.2]{Amb04})
and since ${\rm deg}M_Z=0$ by ~\cite{Takao} it defines a local complex subsystem of the variation of Hodge structure
$R^d h_{\ast}\mathbb{C}$.
By Corollary \ref{gabolomixed} there exists $m$ such that $mM_Z\cong\mathcal{O}_Z$
where $$m=m(b)={\rm lcm}\{k|\phi(k)\leq b\}$$
with $\phi$ the Euler function
and $b=h^d(E,\mathbb{C})$.\\
\item[(ii)]\underline{Unipotent reduction}:
Consider a finite base change as in Proposition \ref{kawamio}
$$\tau\colon Z'\rightarrow Z$$
such that $R^d h'_{\ast}\mathbb{C}_{V'_0}$ has unipotent monodromies,
where $h'\colon V'\rightarrow Z'$ is the induced morphism.
We have $\tau=\tau_k\circ\ldots\tau_1$ where
$$\tau\colon Z'=Z_{k+1}\xrightarrow{\tau_k} Z_k\ldots Z_2\xrightarrow{\tau_1} Z_1=Z.$$
The morphism $\tau_j$ is a cyclic covering defined by the relation
$$A_j^{\otimes\delta_j}\sim H_j,$$
where $A_j$ is a very ample divisor on $Z_j$.
We know by case \textbf{(i)} that $m(b)M_{Z'}\sim\mathcal{O}_{Z'}$.
By Theorem \ref{bettieffcart} $m(b)M_{Z_k}$ is a Cartier divisor.
We have thus the following isomorphisms:
$$\mathbb{C}\cong H^0(Z_{k+1},m(b)M_{Z_{k+1}})\cong H^0(Z_{k+1},\tau^{\ast}m(b)M_{Z_{k}})\cong H^0(Z_k,m(b)M_{Z_{k}}\otimes \tau_{\ast}\mathcal{O}_{Z_{k+1}}).$$
The second isomorphism is by Proposition \ref{basechange}
and the third follows from the projection formula.
From the general theory about cyclic covers we have the isomorphism
$$\tau_{\ast}\mathcal{O}_{Z_{k+1}}\cong \bigoplus_{l=0}^{\delta_k-1} A_k^{-l}.$$
Then we obtain $$H^0(Z_{k+1},m(b)M_{Z_{k+1}})\cong\bigoplus_{l=0}^{\delta_k-1} H^0(Z_k,m(b)M_{Z_{k}}\otimes A_k^{-l}).$$
Since $M_Z\equiv 0$, the divisor $m(b)M_{Z_{k}}\otimes A_k^{-l}$ has negative degree on $Z_k$ for all $l<0$,
thus $$\mathbb{C}\cong H^0(Z_{k+1},m(b)M_{Z_{k+1}})\cong H^0(Z_k,m(b)M_{Z_{k}})$$
and $m(b)M_{Z_{k}}\sim \mathcal{O}_Z.$
We can conclude by induction on $k$.
\end{description}
\end{proof}

\begin{oss}
{\rm
Note that the same proof as in point \textbf{(ii)} of the proof
of Theorem \ref{mainteo} implies a statement on \textbf{Effective non-vanishing} (see also Remark \ref{effeff}).
More precisely assume that $\deg M_Z>0$ and $\tau\colon Z'\rightarrow Z$
is as in Proposition \ref{kawamio}.
Assume that $H^0(Z',m(b)M_{Z'})\neq 0$.
Then the reasoning above implies that $H^0(Z,m(b)M_Z)\neq 0$.
}
\end{oss}

\subsection{LC-trivial fibrations with numerically trivial moduli part.}
In this section we prove Theorem \ref{mainteolc}.
To do so, we prove that $M_Z$ is a subsystem of a variation of Hodge structure
related to the variation of mixed Hodge structure on $R^d\mathbb{C}_{V_0\backslash P_V/Z_0}$.
We start with the following two results.
\begin{prop}\label{trivcurvetrivtuttolc}
Assume that for any lc-trivial fibration
$f\colon (X,B)\rightarrow Z$ with
\begin{itemize}
\item $\dim Z=1$
\item $M_Z\equiv 0$;
\end{itemize}
there exists an integer number $m$
such that $mM_Z\sim\mathcal{O}_Z$.

Then the same holds for bases $Z$ of arbitrary dimension.
\end{prop}
\begin{proof}
The proof follows the same lines as the proof of Proposition \ref{trivcurvetrivtuttoklt}.
\end{proof}
\begin{lem}[Lemma 21, ~\cite{Kaw81}]\label{kawa21}
Let $\mathcal{L}$ be an invertible sheaf over a non-singular
projective curve $C$, let $C_0$ be an open subset of $C$
and let $h$ be a metric on $\mathcal{L}|_{C_0}$.
Let $p$ be a point of $C\backslash C_0$
and let $t$ be a local parameter of $C$ centered at $p$.
We assume that for a uniformizing section $v$ of $\mathcal{L}$
we have $h(v,v)=O(t^{-2\alpha_p}|\log t|^{\beta_p})$
for some real numbers $\alpha_p,\beta_p$.
Then $$\deg_C\mathcal{L}=\frac{i}{2\pi}\int_{C_0}{\Theta}+\sum _{p\in C\backslash C_0}\alpha_p$$
where $\Theta$ is the curvature associated to $h$.
\end{lem}

The following is a generalization of ~\cite[Prop 3.4]{Amb05}
\begin{prop}\label{subsystem}
Let $h\colon V\rightarrow Z$ be a fibration
and let $\Sigma_Z$ be a simple normal crossing divisor such that
\begin{itemize}
\item $h$ is smooth over $Z_0=Z\backslash\Sigma_Z$,
\item $\mathcal{W}_l/\mathcal{W}_{l-1}$ has unipotent monodromies,
where $\{\mathcal{W}_k\}$ is the weight filtration.
\end{itemize}
Let $\mathcal{L}$ be an invertible sheaf such that $\mathcal{L}|_{Z_0}$ 
is a direct summand of $\mathcal{W}_l/\mathcal{W}_{l-1}$
for some $l$.
Assume that $\mathcal{L}\equiv 0$.
Then $\mathcal{L}|_{Z_0}$ is a local subsystem of $\mathcal{W}_l/\mathcal{W}_{l-1}$.
\end{prop}
\begin{proof}
Since $\mathcal{W}_l/\mathcal{W}_{l-1}$ is a variation of Hodge structure,
there is on it a flat bilinear form $Q$ and thus a metric and a metric connection.
Then $\mathcal{L}|_{Z_0}$ has an induced hermitian metric $h$, a metric connection and a curvature $\Theta$.
To prove that $\mathcal{L}|_{Z_0}$ defines a flat subsystem of $\mathcal{W}_l/\mathcal{W}_{l-1}$
it is sufficient to prove that the induced metric connection is flat,
i.e. that $\Theta=0$.\\
The relation between the matrix $\Gamma$ of the connection and the matrix $H$ that represents the metric
is $\Gamma=\bar{H}^{-1}\partial \bar{H}$.
By Remark \ref{lamatrdellaconn}, the order near $p$
of the elements of $\Gamma$ is $O(|t|^{-1}|\log t|^{\beta_p})$.
Let $v$ be a uniformizing section of $\mathcal{L}$.
Then the order of $h(v,v)$ near $p$ is $O(|\log t|^{\beta_p})$.
Let $C\subseteq Z$ be a curve such that $C\cap Z_0\neq\emptyset$.
Let $\nu\colon\hat{C}\rightarrow C$ be its normalization and $C_0=C\cap Z_0$.
We apply Lemma \ref{kawa21} and we obtain
$$\deg_C \mathcal{L}=\frac{i}{2\pi}\int_{\nu^{-1}C_0}{\nu^{\ast}\Theta}.$$
Since $\mathcal{L}$ is numerically zero, we obtain
$$\int_{\nu^{-1}C_0}{\nu^{\ast}\Theta}=0$$
for every $C$ and therefore $\Theta=0$.
\end{proof}
Let us recall that we are working with a cyclic covering $V\rightarrow X$ of degree $r$,
 whose Galois group is the group of $r$-th roots of unity
that we denote $\mu_r$.
Moreover we have an induced action of $\mu_r$ on the sheaves of relative differentials.
Let $P_V$ be the divisor given by the sum of the horizontal places of the pair $(V,B_V)$.

\begin{lem}\label{preservs}
The action of $\mu_r$ on $(h_0)_{\ast}\omega_{V_0/Z_0}(P_V)$ preserves the weight filtration $$\{\mathcal{W}_k((h_0)_{\ast}\omega_{V_0/Z_0}(P_V))\}.$$
\end{lem}
\begin{proof}
A generator of $\mu_r$ determines a birational map $\sigma\colon V\dasharrow V$.
Consider a resolution of $\sigma$
$$
\xymatrix{
V_1 \ar[d]_{\sigma_1} \ar[rd]^{\sigma_2}
&\\
V \ar@{-->}[r]_{\sigma}
& V. }
$$
Let us consider the weight filtrations on $\Omega_{V_0/Z_0}^{\bullet}(\log P_V)$
and on $\Omega_{V_{10}/Z_0}^{\bullet}(\log P_{V_1})$
where $$P_{V_1}={\rm Supp}(\sigma_2^{-1}P_V).$$
The morphism $\sigma_2$ induces, for all $m$, the following:
$$\sigma_2^{\ast}\colon \Omega_{V_0/Z_0}^m(\log P_V)\rightarrow\Omega_{V_{10}/Z_0}^m(\log P_{V_1}).$$
We want to prove that $\sigma_2$ preserves the weight filtration.
Since $\sigma_2$ is a composition of blow-ups of smooth centres it is sufficient to prove the property
for one blow-up.\\
Let $z_1,\ldots,z_n$ be a system of coordinates on $U\subseteq V_0$ such that
$$P_V\cap U=\{z_1\cdot\ldots\cdot z_k=0\}.$$
Then $$\sigma_2^{\ast}\left(\frac{d z_1}{z_1}\wedge\ldots\wedge\frac{d z_h}{z_h}\wedge dz_{h+1}\ldots\wedge d z_n\right)=
\sigma_2^{\ast}\left(\frac{d z_1}{z_1}\right)\wedge\ldots\wedge\sigma_2^{\ast}\left(\frac{d z_h}{z_h}\right)\wedge \sigma_2^{\ast}(dz_{h+1})\ldots\wedge \sigma_2^{\ast}(d z_n).$$
Let $C$ be the centre of the blow-up, let $z_i$ be one of the coordinates.
There are two cases
\begin{description}
\item[(i)] locally $C$ is contained in the zero locus of $z_i$;
\item[(ii)] $C$ is not contained in the zero locus of $z_i$.
\end{description}
In case \textbf{(i)}, let $t$ be an equation of the exceptional divisor and
let $z'_i$ be an equation of the strict transform of $z_i$.
Then $\sigma_2^{\ast}(d z_i)=d(z'_i\cdot t)=t\cdot d z'_i+z'_i\cdot d t$
and $$\sigma_2^{\ast}\left(\frac{d z_i}{z_i}\right)=\frac{t\cdot d z'_i+z'_i\cdot d t}{z'_i t}=\frac{d z'_i}{z'_i}+\frac{d t}{t}.$$
In case \textbf{(ii)}, we simply have
$$\sigma_2^{\ast}(d z_i)=d z'_i\;\;{\rm and}\;\;\sigma_2^{\ast}\frac{d z_i}{z_i}=\frac{d z'_i}{z'_i}$$

Finally, the morphism $\sigma_1$ acts by pushforward and that does not increase the number of poles.
\end{proof}

\begin{proof}[Proof of Theorem \ref{mainteolc}]
By Proposition \ref{trivcurvetrivtuttolc} we can assume that the base $Z$ is a curve.\\
Let us suppose that $\mathcal{W}_k/\mathcal{W}_{k-1}$
has unipotent monodromies for every $k$.
Since by Lemma \ref{preservs} the action of $\mu_r$
preservs the weight filtration
$$\{\mathcal{W}_k((h_0)_{\ast}\omega_{V_0/Z_0}(P_V))\}$$
on $V$, then
for all $k$ the sheaf 
$$\mathcal{W}_k((h_0)_{\ast}\omega_{V_0/Z_0}(\log P_V))=h_{\ast}(\Omega_{V_0/Z_0}^k(\log P_V)\otimes\Omega_{V_0/Z_0}^{n-k})$$ 
decomposes as sum of eigensheaves.
In particular, since $\mathcal{O}_Z(M_Z)$ is an eigensheaf of rank one of $h_{\ast}\omega_{V/Z}(P_V)$,
there exist $l$ such that $\mathcal{O}_Z(M_Z)|_{Z_0}\subseteq {W}_l$
and $\mathcal{O}_Z(M_Z)|_{Z_0}\not\subseteq \mathcal{W}_{l-1}$.
Thus there exists $\mathcal{V}$ containing $\mathcal{W}_{l-1}$
such that $\mathcal{W}_l=\mathcal{O}_Z(M_Z)|_{Z_0}\oplus\mathcal{V}$.
By Proposition \ref{subsystem} thus $\mathcal{O}_Z(M_Z)|_{Z_0}$
defines a local subsystem of $\mathcal{W}_l/\mathcal{W}_{l-1}$.
By Corollary \ref{gabolomixed} we have $\mathcal{O}_Z(M_Z)|_{Z_0}^{m(h)}\sim\mathcal{O}_{Z_0}$
with $h={\rm rk} \mathcal{W}_l/\mathcal{W}_{l-1}$.
Since if the monodromies are unipotent the canonical extension commutes with
tensor product, we have $\mathcal{O}_Z(M_Z)^{m(h)}\sim\mathcal{O}_{Z}$.

The general situation, when $\mathcal{W}_k/\mathcal{W}_{k-1}$
has not unipotent monodromies for every $k$,
can be reduced to the unipotent situation.
We take a covering $\tau\colon Z'\rightarrow Z$ such that on $Z'$ we have unipotent monodromies.
Then $m(h)M_{Z'}\sim \mathcal{O}_{Z'}$ and since $M_{Z'}=\tau^{\ast}M_Z$ we have 
$\deg \tau \cdot m(h) M_Z\sim\mathcal{O}_Z$.
\end{proof}

\end{document}